\documentclass{amsart}

\newtheorem{Lemma}{Lemma}
\newtheorem{Prop}{Proposition}

\newtheorem{Remark}{Remark}

\newcommand{\beq}{\begin{equation}}
\newcommand{\eeq}{\end{equation}}
\newcommand{\bez}{\nopagebreak\hspace*{\fill}%
		 \nolinebreak$\qed$\vspace{5mm}\par}

\newcommand{\C}{\mathbb C}
\newcommand{\N}{\mathbb N}
\newcommand{\R}{\mathbb R}
\newcommand{\T}{\mathbb T}

\let\cal=\mathcal
\newcommand{\ca}{{\cal A}}
\newcommand{\cb}{{\cal B}}
\newcommand{\cc}{{\cal C}}
\newcommand{\cd}{{\cal D}}

\newcommand{\cf}{{\cal F}}

\newcommand{\cp}{{\cal P}}
\newcommand{\ct}{{\cal T}}

\newcommand{\vep}{\varepsilon}
\newcommand{\la}{\lambda}

\newcommand{\si}{\sigma}

\newcommand{\ot}{\otimes}

\let\hat=\widehat

\newcommand{\xbm}{(X,{\cal B},\mu)}
\newcommand{\zdr}{(Z,{\cal D},\rho)}
\newcommand{\ycn}{(Y,{\cal C},\nu)}

\begin{document}

\title{Joining primeness and
disjointness from infinitely divisible systems}

\author{Mariusz Lema\'nczyk\and Fran\c{c}ois Parreau\and Emmanuel Roy}
\thanks{The research of the first author is partially supported by
Polish MNiSzW grant N N201 384834, Marie Curie ``Transfer of
Knowledge" EU program -- project MTKD-CT-2005-030042 (TODEQ), the
research of the first and third authors are partially supported by
MSRI (Berkeley), program ``Ergodic Theory and Additive Combinatorics"}

\address[Mariusz Lema\'nczyk]{ Faculty of Mathematics and Computer
Science, Nicolaus Copernicus University, ul.\ Chopina 12/18, 87--100
Toru\'n, Poland}
\email{mlem@mat.uni.torun.pl}

\address[Fran\c{c}ois Parreau \& Emmanuel Roy]{Laboratoire d'Analyse,
G\'eom\'etrie et Applications -- UMR 7539 Universit\'e Paris 13 et
CNRS, 99 av.  J.-B.\ Cl\'ement, 93430 Villetaneuse, France}
\email{parreau@math.univ-paris13.fr\\ roy@math.univ-paris13.fr}

\date{}

\begin{abstract} We show that ergodic dynamical systems generated by
infinitely divisible stationary processes are disjoint in the
sense of Furstenberg with distally simple systems and systems
whose  maximal spectral type is singular with respect to the
convolution of any two continuous measures.\end{abstract}


\maketitle

\section*{Introduction}
The notion of disjointness of dynamical systems introduced by
Furstenberg \cite{Fu} in 1967 leads to a natural program of
identifying disjoint classes of dynamical systems.  We refer to the
recent monograph \cite{Gl} and also to \cite{Th} for the joining
theory of dynamical systems.  In this note we introduce the notion of
\emph{joining prime} (JP) system.  Such a system $T$ has the property
that every ergodic joining $\la$ of $T$ with the Cartesian product
$S_1\times\ldots \times S_k$ of weakly mixing systems is the direct
product of a joining of $T$ and some $S_i$ with the remaining factors
$S_j$.  In particular, a weakly mixing JP system does not admit a
Cartesian product representation, and it is natural to expect that it
will be disjoint from classes of systems which are sufficiently
divisible in some sense.  Indeed, we will show that the JP class is
disjoint from a class of systems which includes systems of
probabilistic origin, namely those arising from infinitely divisible
(ID) stationary processes, and also from infinitely divisible systems
in the ergodic theoretical sense (to avoid confusion, we will refer to
the latter as \emph{ergodically ID systems}).

The JP class includes two recently considered classes of
automorphisms: (A)~some generalizations of simple systems
 (\emph{distally simple systems}), considered
by Ryzhikov and Thouvenot in \cite{Ry-Th}, and later in \cite{Ju-Le1};
(B)~systems whose maximal spectral type is singular with respect to
the convolution of any two continuous measures on the circle.

Simple systems were introduced by Veech \cite{Ve} and del Junco and
Rudolph \cite{Ju-Ru}.  We recall here that this class includes some
horocycle flows \cite{Ra}, and in fact all horocycle flows are factors
of simple systems \cite{Th}.  Another subclass of simple systems comes
from the finite rank theory (see \cite{Ju-Ra-Sw}, \cite{Ju-Ru1}).  In
\cite{Ju-Le}, simple systems were shown to be disjoint from Gaussian
systems; earlier, in \cite{Th}, Thouvenot showed disjointness of
Gaussian systems from systems with the minimal self-joining
property.
This disjointness result was generalized in two directions.  In
\cite{Th},
Thouvenot introduced ergodically ID systems, extending some properties
of Gaussian systems: an ergodically ID system possesses an infinite
tree of splitting factors, in which for every path the intersection of
the decreasing family of factors is trivial.

Besides a larger class of systems was recently considered in
\cite{Ry-Th},
namely those systems for which all ergodic self-joining different
from the
product measure are relatively distal over the marginal factors.
Disjointness of such systems from some subclasses of ergodically ID
systems including Gaussian systems was then proved in \cite{Ry-Th} and
\cite{Th1}.  When we assume additionally the
so called PID property (each pairwise independent self-joining of such
a system is independent), then we get the class of so-called
\emph{distally simple} systems which are shown to be
disjoint with all ergodically ID systems in \cite{Ju-Le1}.  All flows
with the Ratner
property (\cite{Ra}) are distally simple.  Also, partially mixing
finite rank transformations lie in the class of distally simple
systems \cite{Ki-Th}, \cite{Th}.  Moreover, recently in
\cite{Fr-Le}, it is proved that some smooth flows on surfaces satisfy
the Ratner property and hence are distally simple.

On the other hand, the JP class includes those systems whose maximal
spectral type is singular with respect to the convolution of any two
continuous measures, property which we refer to as the \emph{CS
property}.  A stronger property is that the Gaussian automorphism
constructed from the (reduced) maximal spectral type has a simple
spectrum (\emph{simple convolutions property}).  The famous Chacon
automorphism has both the minimal self-joinings property
(\cite{Ju-Ra-Sw}) and the simple convolution property (\cite{Ag1}).
Ageev \cite{Ag2} and Ryzhikov \cite{Ry2} recently showed that many
other rank one transformations enjoy the simple convolution property,
including typical automorphisms and mixing examples.  These properties
are also related to the property that the convolution powers of the
(reduced) maximal spectral type are mutually singular (see \cite{Ka},
and also \cite{Ju-Le0} and \cite{Th} for usefulness of such a property
in ergodic theory).

The extension of the definitions and results to flows and actions of
other locally compact abelian groups is straightforward.
In Section~\ref{flows}, we will show that when $\ct=(T_t)_{t\in\R}$
is a measurable flow and
in the weak closure of the unitary group coming from time~$t$
automorphisms we find an operator of the form $f(\ct)$, where
$f:\R\to\C$ is analytic and different from a multiple of a character
then the flow has the CS property, and hence it has also the JP
property.  The assumption turns out to be satisfied for many natural
examples of flows (see Section~\ref{flows}).  In
\cite{Le-Pa}, even the simple convolutions property is shown to hold
for a large class of smooth flows on the torus.

We show that systems with the JP property are disjoint with
ergodically ID systems, which extends the result of \cite{Ju-Le1} for
distally simple systems.  In particular a consequence is that a
typical dynamical system is disjoint from every ergodically ID system.
In fact we give a general criterion (Proposition~\ref{l3}) for such a
disjointness which is apparently satisfied for a larger class than
ergodically ID systems.  The criterion applies to systems which can be
embedded in symmetric factors of infinitely many Cartesian products,
and hence is especially adapted to systems given by stationary ID
processes (and in particular to Poisson suspension systems).  However,
we also give simple arguments showing that systems generated by
stationary ID processes, like Poisson suspensions, are factors of
ergodically ID systems.

Most results of this paper have been obtained during the visit the
first
author at the University Paris 13 in Fall 2006.

\section{Basic definitions and notation}\label{Def}
We will firstly consider automorphisms acting on standard probability
Borel spaces
and their ergodicity will be tacitly
assumed.  Joinings between automorphisms $T$ and $S$ are meant as
measures on the corresponding product space having the coordinate
projections as the original measures and invariant under $T\times S$.
The set of joinings between these automorphisms is denoted by $J(T,S)$
(with an obvious change of notation when more automorphisms are
involved).  The subset of ergodic joinings will be denoted by
$J^e(T,S)$.  Following \cite{Fu}, systems $T$ and $S$ between which
the only possible joining is the product measure are called {\em
disjoint}; we write then $T\perp S$.  If $T$ and $S$ act on $\xbm$ and
$\ycn$ respectively and if $\lambda\in J(T,S)$ then it determines a
Markov operator $\Phi_{\lambda}:L^2\xbm\to L^2\ycn$,
\beq\label{markov1}
    \int_Y\Phi_\lambda(f)g\,d\nu =
    \int_{X\times Y}f\otimes g\,d\lambda,\quad\Phi_\lambda\circ
T=S\circ\Phi_\lambda
\eeq
(we still denote by $T$ and $S$ the corresponding unitary operators
on the $L^2$ spaces).
That is $\Phi_\lambda$ is doubly stochastic:
$\Phi_\lambda(1)=\Phi_\lambda^\ast(1)=1$, and the image via
$\Phi_{\lambda}$ of a non-negative function is non-negative.  Notice
that $\Phi_\lambda(f)=E^\lambda(f \otimes 1|Y)$, where by
$E^{\lambda}(\cdot|Y)$ we denote the conditional expectation with
respect to the $\sigma$-algebra $\{\emptyset,X\}\otimes\cc$.  In
fact~(\ref{markov1}) establishes a 1-1 correspondence between joinings
and Markov operators, see e.g. \cite{Gl}, Chapter~6 for more details.

Let $R$ be an ergodic automorphism of a standard probability Borel
space $\zdr$.  Following \cite{Ju-Ru}, $R$ is said to have the {\em
pairwise independence property} (PID) if every (finite or infinite)
pairwise independent self-joining of $R$ is equal to the product
measure.  Ryzhikov in \cite{Ry} proved the following important lemma.

\begin{Lemma}\label{lryz}
Assume that $R$ has the PID property and let $T$, $S$ be
automorphisms.
Assume that $\lambda\in J(T,S,R)$ is pairwise independent. Then it is
the
product measure.
\end{Lemma}

For the definition of distality and relative distality we refer the
reader to \cite{Fu1}.  An automorphism $R$ is called {\em 2-fold
distally simple} (see \cite{Ju-Le1}, \cite{Ry-Th}) if each $\la\in
J^e(R,R)$ is either the product measure, or the extension $(R\times
R,\la)$ is relatively distal over the marginal factors $R$.  A 2-fold
distally system which has the PID property is called {\em distally
simple}.

As proved in \cite{Ju-Le1}, each 2-fold distally simple
automorphism  either has purely discrete spectrum or is weakly
mixing. Another important property of a 2-fold distally simple
system $R$ is that (see \cite{Ju-Le1}, and also
\cite{Ry-Th}):
\beq\label{jule}
\parbox[c]{11.5cm}{\em if $S$ is an arbitrary automorphism and
$\la\in J^e(R,S)$
	then either $\la$ is the product measure or the extension
	$(R\times S,\la)$ is relatively distal over the factor $S$.}
\eeq

We also recall (see \cite{Fu}) that

\beq\label{furst}
\parbox[c]{11.5cm}{\em if $T$ is ergodic, the extension $T\to T_1$ is
relatively distal and 
$S$ is weakly mixing, then the only extension of $T_1\times S$ to an
ergodic joining of $T$
and $S$ is the direct product.}
\eeq

A dynamical system $T$ will be said to have the {\em convolution
singularity}
(CS) property  if its reduced maximal
spectral type $\sigma_T$ (i.e. the maximal spectral type of $T$ on
$L^2_{0}\xbm$,
the orthocomplement of the constant functions) is singular with
respect to the convolution
product $\la_1\ast\la_2$ of two arbitrary continuous
measures $\la_1,\la_2$ on the circle. CS systems have purely
singular spectra, thus they automatically have the PID property by
the result of Host
\cite{Ho}. It can be  seen (\cite{Pa-Ro}) that the CS
property holds if the symmetric tensor product operator
$V_\sigma\odot V_\sigma$ has simple spectrum, where
$$V_\sigma:L^2(\T,\sigma)\to
L^2(\T,\sigma),\;\;V_\sigma(f)(z)=zf(z).$$
This relates CS property of $T$ with spectral properties of
the Gaussian automorphism determined by $\sigma_T$. In particular
if this Gaussian automorphism has a simple spectrum then $T$ has
CS property (see \cite{Co-Fo-Si}).

In what follows, if no confusion arises, we will identify
automorphisms with their actions on $\sigma$-algebras, in particular
factor automorphisms will be identified with the underlying invariant
sub-$\sigma$-algebras.

\section{Systems  with joining primeness property and
disjointness}\label{disjJP}

An ergodic system $R$ acting on $\zdr$ is said to have the
\emph{joining primeness} (JP)
property if for every pair of weakly mixing systems $S_{1}$ and
$S_{2}$
and every $\lambda\in J^{e}\left(R,S_{1}\times S_{2}\right)$ we have,

\[
\lambda=\lambda_{Z,Y_{1}}\otimes\nu_{2}\quad\textrm{or}\quad\lambda=\lambda_{Z,Y_{2}}\otimes\nu_{1},
\]
where by $\lambda_{Z,Y_{i}}$ we denote the projection of $\lambda$
to the corresponding two coordinates.

Of course this definition really makes sense only for a system which
is not disjoint from a
weakly mixing system.

Notice that assuming $S_{1}$ and $S_{2}$ isomorphic would give an
equivalent definition. Indeed, let $\lambda\in J^{e}\left(R,S_1\times
S_2\right)$
where $S_1$ and $S_2$ are two weakly mixing systems, not necessarily
isomorphic. If $S^{\prime}_1$ and $S^{\prime}_2$ are isomorphic
copies of $S_1$ and $S_2$ respectively, then, by forming the direct
product of the previous joining
with the direct product $S^{\prime}_1\times S^{\prime}_2$, we obtain
an ergodic joining of $R$ with the direct product
$\left(S_1\times S^{\prime}_2\right)\times\left(S^{\prime}_1\times
S_2\right)$
of two isomorphic weakly mixing transformations and the equivalence
of the definition follows. Remark also that we could have required
$R$ to satisfy a seemingly stronger property: for any $n\geq2$,
any familly $S_{1},\dots,S_{n}$ of weakly mixing transformations and
any
$\lambda\in J^e(R,S_1\times\ldots\times S_n)$,
there exists $j(\lambda)\in\{1,\ldots, n\}$ such that
\beq\label{jlambda}
\lambda=\lambda_{Z,Y_{j(\lambda})}\otimes\left(\otimes_{j\neq
j(\lambda)}\nu_{j}\right).
\eeq
Once again, it is not difficult to check that the resulting definition
would be equivalent to the previous one and that it is still so if
$S_{1},\dots,S_{n}$
are assumed isomorphic.

Given $\la\in J^e(R,S_1\times\ldots\times S_n)$, we identify $\cd$
and the $\cc_i$
to sub-$\sigma$-algebras of the joining system. It is easy to see
that $\lambda$
satisfies (\ref{jlambda}) if and only if
\beq\label{f4}
\Phi_{\lambda}(f)\in L^2(\cc_{j(\lambda)})
\eeq for each $f\in L^2_0\zdr$.

\begin{Prop}
The class of weakly mixing JP systems is closed under factors,
inverse limits and distal extensions which are weakly mixing.
\end{Prop}
\begin{proof}
The  first two properties are obvious. The fact that the JP
property is closed under weakly mixing distal extensions follows from
(\ref{furst}).
\end{proof}

Let us notice that whenever a JP automorphism $R$ has a common
factor with a direct (weakly mixing) product $S_1\times S_2$ then
in fact this common factor is contained in either the first or the
second coordinate $\sigma$-algebra of $\cc_1\ot\cc_2$. Hence using
del Junco-Rudolph's criterion from \cite{Ju-Ru} of
non-disjointness of a system from a simple system we obtain the
following.

\begin{Prop}
A JP system is disjoint with a weakly mixing simple system if and
only if they do not have a common non-trivial factor.\bez
\end{Prop}

The result below is proved in \cite{Ry-Th} and \cite{Ju-Le1} in a
slightly different form 
and we give a proof only for completeness.

\begin{Prop}\label{mar}
If $R$ is weakly mixing and distally simple  then it has the JP
property.
\end{Prop}
\begin{proof}
Assume that $\lambda\in J^e(R,S_1\times S_2)$. By (\ref{jule}), the
extension
$(R\times S_1, \lambda_{Z,Y_1})\to S_1$ is either relatively distal
or we have
independence. If it is relatively distal, since $S_2$ is weak mixing,
we have
$\lambda=\lambda_{Z,Y_1}\otimes\nu_2$ by (\ref{furst}).
Similarly, if 
$(R\times S_2, \lambda_{Z,Y_2})\to S_1$ is relatively distal, we get
$\lambda=\lambda_{Z,Y_2}\otimes\nu_1$.
Furthermore, if $\cd$ is independent from $\cc_1$ and $\cc_2$, then
by the PID property of $R$
(and Lemma~\ref{lryz}), $\lambda$ is just the product measure.
\end{proof}

In order to prove that CS implies JP we need a
lemma.

\begin{Lemma}\label{mar1} Assume that
$\lambda \in J^e(R, S_1\times  S_2)$ and it satisfies
$$
\Phi_\lambda(L_0^2\zdr)\subset L_0^2(Y_1,\cc_1,\nu_1) \oplus
 L_0^2(Y_2,\cc_2,\nu_2).$$
Then $\lambda=\lambda_{Z,Y_1} \otimes\nu_2$ or 
$\lambda=\lambda_{Z,Y_2} \otimes\nu_1$.
 \end{Lemma}
 \begin{proof}
 Denoting by $p_i$ the orthogonal projection from $L^2(Y_1\times
Y_2,\nu_1\ot\nu_2)$ on
 $L^2(Y_i,\nu_i)$ and $E$ the orthogonal projection on constant
functions,
 by the assumption
 \beq\label{d1}
 \Phi_\lambda+E=
 p_1\circ
 \Phi_\lambda+p_2\circ\Phi_\lambda.
 \eeq
 But  $p_i\circ \Phi_\lambda$ is still a Markov operator, so
$p_1\circ \Phi_\lambda =
 \Phi_{\lambda_i}$ for some joining $\la_i$, $i=1,2$.
 Now the condition in (\ref{d1}) simply means that
 $$
 \lambda+\mu\ot\nu_1\ot\nu_2=\lambda_{X,Y_1}\ot \nu_2+
\lambda_{X,Y_2}\ot \nu_1
 $$
 and since the four joinings under consideration are ergodic
 (the factor joinings $\lambda_{Z,Y_i}$ are ergodic\
 and the $S_i$ are weakly mixing), the assertion follows.
 \end{proof}

\begin{Prop}\label{mar2}
If $R$ is weakly mixing and has the CS property  then it has the JP
property.
\end{Prop}
\begin{proof}
For each $f\in L^2\zdr$, it is a general fact that the spectral
measure of $\Phi_\lambda(f)$
is absolutely continuous with respect to the spectral measure of $f$
and therefore
with respect to $\sigma_R$. Since
\[
L_0^2(Y_1\times Y_2,\nu_1\ot\nu_2)= L_0^2(Y_1,\nu_1)\oplus
L_0^2(Y_2,\nu_2)\oplus
 \left(L_0^2(Y_1,\nu_1)\otimes L_0^2(Y_2,\nu_2)\right)
\]
and on 
$L_0^2(Y_1,\nu_1)\otimes L_0^2(Y_2,\nu_2)$
the maximal spectral type is the convolution
of the reduced maximal spectral types of $S_1$ and $S_2$, we obtain
the assumption of Lemma~\ref{mar1} satisfied. Therefore, the
result follows.
 \end{proof}

We will give examples of automorphisms with the CS property in
Section \ref{flows}.

\section{Joining primeness property and symmetric factors}\label{ID}

The aim of this section is to prove that JP systems are disjoint from
dynamical systems that possess the probabilistic property of
\emph{infinite
divisibility} (not to be mistaken with the ergodic theoretical notion
sharing the same name, see next section).

Let us recall the definition of an infinitely divisible (ID)
stationary
process:

A stationary process $\left\{ X_{n}\right\} _{n\in\bf{Z}}$ of
distribution $P$ on $\left(\bf{R}^{\bf{Z}},\cb^{\otimes\bf{Z}}\right)$
is ID if, for any integer $k$, there exists a probability
distribution $P_{k}$
on $\left(\bf{R}^{\bf{Z}},\cb^{\otimes\bf{Z}}\right)$ such that
$P=P_{k}*\dots*P_{k}$ ($k$ terms) where $*$ denotes the convolution
operation with respect to
the usual coordinatewise addition on $\bf{R}^{\bf{Z}}$.
In other words, the stationary process $\left\{ X_{n}\right\}
_{n\in\bf{Z}}$ is ID if, 
for any integer $k$, it can be realized as an independent sum
$X_{n}=X_{n}^{\left(1,k\right)}+\dots+X_{n}^{\left(k,k\right)}$ where
the 
processes $\left\{ X_{n}^{\left(i,k\right)}\right\}
_{n\in\mathbb{Z}}$ are stationary
with the same distribution $P_{k}$.
If $P$ is infinitely divisible, then the distributions $P_{k}$ are
uniquely determined
and are also infinitely divisible. Moreover if
$\left({\bf{R}^{\bf{Z}}},\cb^{\otimes\bf{Z}},P,S\right)$ (where $S$
stands for the shift) is ergodic,
then it is weakly mixing (see \cite{Ros}) 
and so is
$\left({\bf{R}^{\bf{Z}}},\cb^{\otimes\bf{Z}},P_{k},S\right)$ for any
$k$.
Any Gaussian stationary process is infinitely divisible as well as any
stable stationary process among many others.

The important ergodic property here that will play the key role for
our disjointness result comes from the following observation:
{\em the system determined by a stationary ID process is, for any
integer $k\ge 1$, a factor of the symmetric
factor of the cartesian $k$-th power of a dynamical system (by
symmetric factor we mean the sub-$\sigma$-algebra of subsets
which are invariant under all permutations of coordinates).}
Indeed, $\left({\bf{R}^{\bf{Z}}},\cb^{\otimes\bf{Z}},P,S\right)$
is a factor of
$\left({\bf{R}^{\bf{Z}}},\cb^{\otimes\bf{Z}},P_{k},S\right)^{\times
k}$
via the factor map
$\left(\left\{ x_{n}^{1}\right\} _{n\in\bf{Z}},\dots,\left\{
x_{n}^{k}\right\}_{n\in\bf{Z}}\right)
\mapsto\left\{ x_{n}^{1}+\dots+x_{n}^{k}\right\}_{n\in\bf{Z}}$.

Closely related to ID stationary processes are Poisson suspensions.
Let us recall the definition.

Consider a dynamical system $\left(X,\ca,\mu,T\right)$ where $\mu$ is
$\sigma$-finite and
form the Poisson measure $\left(X^{*},\ca^{*},\mu^{*}\right)$
where $X^{*}$ is the space of counting measures on
$\left(X,\ca\right)$,
$\ca^{*}$ the smallest $\sigma$-algebra such that for any
$A\in\ca^{*}$
the map $N_{A}:\nu\mapsto\nu\left(A\right)$ is measurable and
$\mu^{*}$ is
the unique probability measure such that for any family
$A_{1},\dots,A_{k}$
of disjoint sets in $\ca$ of finite $\mu$-measure, the random
variables $N_{A_{i}}$ are independent and Poisson distributed with
parameters $\mu(A_{i})$. Let furthermore $T_{*}:\ X^{*}\to X^{*}$ be
the map which assigns to each
$\nu\in X^{*}$ its direct image under $T$.
Then $T_{*}$ is $\mu^{*}$-preserving and
$\left(X^{*},\ca^{*},\mu^{*},T_{*}\right)$
is called the \emph{Poisson suspension} over the \emph{base}
$\left(X,\ca,\mu,T\right)$
(see \cite{Co-Fo-Si}, and also recent \cite{De-Fr-Le-Pa}, \cite{Ro},
\cite{Ro1} for the joining theory of
Poisson suspensions). Convolution in this space is well defined,
arising from
addition of measures. It is well known that
$\mu^{*}=\left(\frac{1}{k}\mu\right)^{*}*\dots*\left(\frac{1}{k}\mu\right)^{*}$
($k$ terms) for every $k\ge 1$, and this shows the ID character of a
Poisson suspension.
We therefore derive the same conclusion as above: for any integer
$k$ a Poisson suspension is a factor of the symmetric factor of a
direct product of $k$ copies of a dynamical system.

The disjointness of JP systems from systems given by ID stationary
processes
and Poisson suspensions is based on the relationship between JP
systems and symmetric factors.
In fact, it will be a consequence of a bit more general criterion
given in
Proposition \ref{l3} below. We first need a lemma.

Let $S$ be a weakly mixing automorphism of a standard probability
space $\ycn$. Given $n\geq1$, by $\cf_n=\cf_n(\cc)=\cf_n(S)$ we
denote the {\em
  symmetric factor} of $S^{\times n}$.

\begin{Prop}\label{l3} Let $T$ be an automorphism of a standard
probability
 space $\xbm$. Assume that, for each
$g\in L^2\xbm$,  there exist a sequence $(k_j)_{j\ge1}$ of integers
going
to infinity and a sequence $(S_j)_{j\ge1}$ of weakly mixing
automorphisms such that $T$ is a factor of $S_j^{\times
k_j}$ and moreover
$$
dist(g,L^2(\cf_{k_j}(S_j))=\mbox{o}\left(\frac1{\sqrt{k_j}}\right)\cdotp$$
Then $T$ is disjoint from every weakly mixing JP system.
\end{Prop}
\begin{proof}
Let $\lambda_0\in J(R,T)$, where $R:\zdr\to\zdr$ is a JP automorphism
and let
$f$ be any function in $L^2_0\zdr$.
We have to prove that $g:=\Phi_{\lambda_0}(f)=0$.

By the assumption, given an arbitrary $\vep>0$, there exist an
arbitrarily large
integer $k>0$  and a weakly mixing automorphism $S$ on a standard
space $\ycn$ and
such that $T$ is a factor of $S^{\times k}$ and moreover
\beq\label{f0}
dist(g,L^2(\cf_{k}(S))\le\frac{\vep}{\sqrt k}\;\cdotp
\eeq
Let $\widetilde{\lambda}_0$ be the relatively independent extension
of $\lambda_0$ to a
joining of $R$ with $S^{\times k}$, so that
$\Phi_{\widetilde{\lambda}_0}$ is the composition
of $\Phi_{\lambda_0}$ and the embedding of $L^2\xbm$ into
$L^2(Y^{\times n},\cc^{\otimes n},\nu^{\otimes n})$, and thus
$\Phi_{\widetilde{\lambda}_0}(f)= \Phi_{\lambda_0}(f)=g$.

Consider now an ergodic decomposition of $\widetilde{\lambda}_0$:
$$
\widetilde{\lambda}_0=\int_{J^e(R,S^{\times k})} \lambda dP(\lambda).
$$
For $\lambda\in J^e(R,S^{\times k})$, we write the underlying
$\si$-algebra as $\cd\vee \cc_{1}\vee\ldots\vee \cc_{k}$,
where the $\cc_{i}$'s algebras remain independent. When
$\Phi_{\lambda}(f)\ne 0$, since $R$ has
the JP property, there exists exactly one $1\leq j=j(\lambda)\leq k$
such that
$\Phi_{\lambda}(f)\in L^2(\cc_{j(\lambda)})$.

For $1\le j\le k$, denote by $M_j$ the set of those $\lambda\in
J^e(R,S^{\times k})$ for which 
$\Phi_{\lambda}(f)\ne 0$ and $j(\lambda)=j$.  The sets $M_j$ are
disjoint and clearly measurable. Let
$$f_j:=\int_{M_j}\Phi_{\lambda}(f)\,dP\in L^2(\cc_j).$$
We obtain the orthogonal decomposition
$$
g=\int_{J^e(R,S^{\times k})} \Phi_{\lambda}(f)\,dP=\sum_{j=1}^k
f_j\in\bigoplus_{j=1}^k L^2(\cc_j)
$$
and moreover there exists some $j_0\in \{1,\ldots, k\}$ such that
$P(M_{j_0})\leq \frac1k$, whence
\beq\label{f7}
\|f_{j_0}\|\leq \frac1k\|f\|.\eeq
Now we claim that for all $1\le j\le k$
\beq\label{f8}
\bigl|\|f_j\|-\|f_{j_0}\|\bigr|\leq\frac{2\vep}{\sqrt k}\;\cdotp
\eeq
Indeed, consider a permutation $\pi$ of $\{1,\ldots ,k\}$ sending
$j_0$ to $j$, and the
corresponding unitary operator $U_{\pi}$ on $L^2(Y^{\times
n},\cc^{\ot n},\nu^{\ot n})$.
Since all functions in $L^2(\cf_k(S))$ are fixed by $U_\pi$,
$$
\|f_{j}-f_{j_0}\|\le\left(\sum_{\ell=1}^k
\|f_{\pi(\ell)}-f_\ell\|^2_{L^2(\cc)}\right)^{1/2}
=\|U_\pi(g)-g\|\leq 2\,\text{dist}(g,L^2(\cf_{k}(S))
$$
and (\ref{f8}) follows from (\ref{f0}).

So, by (\ref{f7}), $\|f_j\|\le \frac1k\|f\|+\frac{2\vep}{\sqrt k}$
for all $1\le j\le k$ and
$$
\|g\|=\left(\sum_{j=1}^k\|f_j\|^2\right)^{1/2}\leq
\frac1{\sqrt k}\|f\|+ 2\vep.
$$
Since $k$ was arbitrarily large and $\vep$ was arbitrarily small,
this proves that $g=0$.
\end{proof}

\begin{Remark}\em
We notice that $T$ satisfying the assumption of Proposition~\ref{l3}
is still disjoint from all roots of JP maps. Indeed, if $T$ satisfies
the assumptions of
Proposition~\ref{l3} then each (non-zero) power of it  does so;
moreover if two automorphisms are non-disjoint then also their
powers are non-disjoint and the disjointness result with roots of
JP maps follows.
\end{Remark}

\section{JP property and infinite divisibility in the ergodic
theoretical sense}\label{ErgoID}

Let us recall that an ergodic automorphism $T$ is said to be {\em
infinitely divisible} if there exists a sequence of factors
$\{\cb_\omega:\:\omega \in\{0,1\}^\ast\}$ of $\cb$ where
$\cb_{\vep}=\cb$,
$\cb_\omega=\cb_{\omega0}\ot\cb_{\omega1}$ and for each $f\in
L_0^2\xbm$, $\eta\in\{0,1\}^{\N}$
\beq\label{d2} \lim_{n\to
\infty}E(f |\cb_{\eta[0,n)})=0.
\eeq

As in the introduction, to avoid confusion, this ergodic theoretical
notion of infinite
divisibility will be referred to as {\em ergodically ID} in the
sequel.

\begin{Prop}\label{mar4} Ergodically ID automorphisms are disjoint
from JP systems.
\end{Prop}
\begin{proof} Assume that $R$ is a JP automorphism, and let $\la\in
J^e(R,T)$,
where $T$ is ergodically ID. For each $f\in L^2_0\zdr$ and $n\geq1$,
by the
JP property of $R$, there is an $\omega\in\{0,1\}^n$ such that
$\Phi_\lambda(f)$ belongs to $L^2(\cb_\omega)$ (indeed,
$\cb=\bigotimes_{|\omega|=n}\cb_\omega$). Then the result follows
immediately from~(\ref{d2}).\end{proof}

In order to complete the picture we will now give a general
argument which in particular shows that dynamical systems induced
by stationary ID processes are factors of ergodically ID dynamical
systems.

\begin{Prop}\label{Fran1}
Let $T$ be an automorphism of a standard probability
space $\xbm$. Assume that there exists a sequence $(S_n)_{n\ge 0}$ of
automorphisms of standard probability spaces $(Y_n,\cc_n,\nu_n)$ such
that $S_0=T$
and, for every $n\ge 0$, $S_n$ is isomorphic to a factor of the
symmetric factor
of $S_{n+1}\times S_{n+1}$. Then $T$ is a factor of an  ergodically
ID system.
\end{Prop}
\begin{proof}
Let, at each stage $n\ge 0$, $T_n:=\Pi_{\omega\in\{0,1\}^n}S_\omega$
be the direct product
of $2^n$ copies $S_\omega$ of $S_n$ acting each on a space
$(Y_\omega,\cc_\omega, \nu_\omega)$,
and $T_n$ acting on the product space $(X_n,\cb_n,\mu_n)$.
From the assumption we represent each $S_\omega$ as a factor of  the
symmetric factor of
$S_{\omega 0}\times S_{\omega 1}$. So each $T_n$ is a factor of
$T_{n+1}$ and we obtain an inverse system of dynamical systems.
Let $\widetilde{T}$ stand for the inverse limit, acting on
$(\widetilde{X},\widetilde{\cb},\widetilde{\mu})$.
We will show that $\widetilde{T}$ is ergodically ID.

We consider each  $\cb_n$ and $\cc_\omega$ as sub-$\sigma$-algebras
of $\widetilde{\cb}$: then $\widetilde{\cb}$ is
generated by the union of the increasing sequence of factors $\cb_n$
and each $\cc_\omega$ is contained in the
symmetric factor $\cf_2(\cb_{\omega0},\cb_{\omega1})$. Let moreover,
for $\omega\in\{0,1\}^\ast$
$$
\cb_\omega=\bigvee_{\eta\in\{0,1\}^\ast}\cc_{\omega\eta}.
$$
We obtain a splitting sequence of
independent factors as required in the definition of an ergodically
ID system.
We need to show that~(\ref{d2}) holds true and it is enough here to
consider
$f\in L_0^2(\cb_n)$ for some $n\ge 0$.
Then, for $\omega\in \{0,1\}^n$,
\beq\label{dec1}
E(f|\cb_\omega)=E(f|\cc_\omega).
\eeq
Indeed, it is again enough here to consider $f$ of the form
$f=\ot_{\omega'\in\{0,1\}^n} f_{\omega'}$ with $f_\omega'$ measurable
with respect to $\cc_{\omega'}$, and then $f_{\omega'}$ is still
independent
from $\cb_\omega$ whenever $\omega'\ne \omega$.

Now, represent $E(f|\cc_\omega)$ as an orthogonal sum
$$
E(f|\cc_\omega)=E(f|\cc_{\omega0})+E(f|\cc_{\omega1})+g,$$
Since $\cc_\omega$ is contained in the symmetric factor of
$S_{\omega0}\times S_{\omega1}$,
we have that $E(f|\cc_{\omega1})$ corresponds to $E(f|\cc_{\omega0})$
by the given isomorphism
between $S_{\omega0}$ and $S_{\omega1}$, and in particular it has the
same norm.
Therefore
$$\|E(f|\cc_{\omega i})\|\le \frac{1}{\sqrt2}\|E(f|\cc_{\omega})\|\le
\frac{1}{\sqrt2}\|f\|$$
for $i=0,1$.

We may repeat the same argument and thus we
obtain for every $\eta\in\{0,1\}^\ast$
$$\|E(f|\cc_{\omega\eta})\|
\leq\left(\frac{1}{\sqrt{2}}\right)^{|\eta|}\|f\|.$$
Moreover, since $\cb_n\subset \cb_{p}$ for $p\ge n$, we still have
$E(f|\cb_{\omega\eta})=E(f|\cc_{\omega\eta})$  by (\ref{dec1}) and
the result follows.
\end{proof}

Applying this result to a Poisson suspension
$\left(X^{*},\ca^{*},\mu^{*},T_{*}\right)$
and following the lines of the construction, it can be checked
that the inverse limit is indeed isomorphic to the Poisson suspension
of
\[
\left(X\times[0,1),\;\ca\otimes\mathcal{B}([0,1)),\;\mu\otimes\lambda,\;T\times
Id\right)
\]
where $\lambda$ is Lebesgue measure on $[0,1)$. Indeed, we recall
that  $\left(X^{*},\ca^{*},\mu^{*},T_{*}\right)$
is a factor of the symmetric factor of
$\left(X^{*},\ca^{*},\left(\frac1{2}\mu\right)^{*},T_{*}\right)^{\times
2}$
but this latter system is isomorphic to the Poisson suspension of
\[
\left(X\times\left\{ 0,1\right\},\;\ca\otimes\cp\left(\left\{
0,1\right\} \right),\;
\mu\otimes
\left(\frac{1}{2}\delta_{0}+\frac{1}{2}\delta_{1}\right),\;T\times
Id\right)
\]
and the result follows, proceeding inductively.

\section{Remarks concerning JP property of higher
order}\label{HigherJP}

We can extend the definition of JP, considering a weaker property for
multiple joinings:
an ergodic system $R$ is said to have \emph{the joining primeness
property of order} $n\geq1$ (JP($n$)) if, for every $k\geq n+1$,
every direct product of $k$
weakly mixing automorphisms $S_{1},\ldots,S_{k}$, for every
$\lambda\in J^{e}\left(R,S_{1}\times\ldots\times S_{k}\right)$,
there exist $i_{1},\dots,i_{n}$ in $\{1,\ldots, k\}$ such that
\[
\lambda=\lambda_{Z,Y_{i_1},\ldots,Y_{i_n}}\otimes\bigotimes_{j\neq
i_{1},\dots,i_{n}}\nu_j.
\]

Clearly JP($m$)$\subseteq$JP($n$) whenever $m\leq n$ and JP($1$)
is just the class JP considereded in this paper. Note that, as in
this latter
case, we can replace in the definition possibly non isomorphic
$S_{i}$'s by
isomorphic ones.

\begin{Prop}
JP($n$) systems are disjoint from ergodically ID systems.
\end{Prop}
\begin{proof}
Let $R$ belong to the JP($n$) class, $T$ be ergodically ID and
consider
an ergodic joining $\lambda$ of $R$ and $T$ and the corresponding
Markov operator
$\Phi_{\lambda}:\;L^2\zdr\to L^2\xbm$.

With the notation of previous section, let $\Omega_{k}$ be the set of
all $\omega\in\{0,1\}^{\ast}$
with $\left|\omega\right|=k$, for $k\ge 0$, so that
$\cb=\bigotimes_{\omega\in\Omega_{k}}\mathcal{B}_{\omega}$.
By the JP($n$) property of $R$, we can find subsets $\Delta$ of
$\Omega_{k}$ of cardinality${}\le n$ such that
$\lambda$ is the product of its restriction to $\cd\vee
\bigotimes_{\omega\in\Delta}\cb_{\omega}$ and the
other factors, or equivalently
\beq\label{JPn}
\Phi_{\lambda}(L^2(\cd))\subset
L^2\left(\bigotimes_{\omega\in\Delta}\cb_{\omega}\right).
\eeq
Since, for $\Delta,\, \Delta'\subset \Omega_{k}$, we have by
independence 
\[
 L^2\left(\bigotimes_{\omega\in\Delta}\cb_{\omega}\right)\cap
L^2\left(\bigotimes_{\omega\in\Delta'}\cb_{\omega}\right)=
  L^2\left(\bigotimes_{\omega\in\Delta\cap\Delta'}\cb_{\omega}\right),
\]
it follows that there is a smallest set $\Delta\subset\Omega_{k}$
satisfying (\ref{JPn}), which we denote by
$\Omega_{k}^{\lambda}$. Let also  $n_{k}$ be its cardinality ($n_k\le
n$ for all $k$).

Now, from $\cb_{\omega}=\cb_{\omega0}\otimes\cb_{\omega1}$ for all
$\omega$, it follows that  $\Omega_k^{\lambda}$ is contained in the
set of all $\omega[0,k)$ for $\omega\in\Omega_{k+1}^{\lambda}$, and
conversely $\Omega_{k+1}^{\lambda}$ is contained in the set of all
$\omega0$ and all $\omega1$ for $\omega\in\Omega_{k}^{\lambda}$. In
particular, $n_{k}$ is non decreasing and, being bounded by $n$, it
is eventually constant. This implies that there exists $K$ such that
for $k\ge K$ $n_k=n_K$ and for each $\omega\in\Omega_{k}^{\lambda}$
there is exactly one $\omega'\in
\{\omega0,Ê\omega1\}$ in $\Omega_{k+1}^{\lambda}$. Hence there exist
$\eta_1,\ldots \eta_{n_K}$ in $\{0,1\}^{\N}$ such that
\[
\Phi_{\lambda}(L^2(\cd))\subset \bigcap_{k\ge
K}L^2\left(\bigotimes_{i=1}^{n_K}\cb_{\eta_i[0,k)}\right).
\]

Finally, again by independence, for each $k\ge K$, the orthogonal
projection from
$L^2\left(\bigotimes_{i=1}^{n_K}\cb_{\eta_i[0,K)}\right)$
onto $L^2\left(\bigotimes_{i=1}^{n_K}\cb_{\eta_i[0,k)}\right)$ is the
tensor product of the orthogonal projections
of the factors. By (\ref{d2}) for each $i\in\{1,\ldots, n_K\}$ the
projection of every $f\in L^2_0(\cb_{\eta_i[0,K)})$ in
$L^2(\cb_{\eta_i[0,k)})$
goes to zero as $k$ tends to infinity and it follows that the same
holds for the tensor product. Therefore 
$\Phi_{\lambda}$ is the orthogonal projection onto the constant functions
and $\lambda$ is the product measure.
\end{proof}

\subsection{Examples}

We present here a family of systems which are JP($2$) but not
JP($1$). Recall that a direct product of
two weakly mixing automorphisms cannot have the JP property, so we
get such examples from the Proposition below
if we assume moreover the automorphisms to be weakly mixing.
\begin{Prop}\label{jp2}
If $R_1$ has the JP property and $R_2$ is distally simple then
$R_1\times R_2$ is
JP($2$).
\end{Prop}
\begin{proof}

Let $S_{1}$, $S_{2}$ and $S_{3}$ be weakly mixing automorphisms and
consider an
ergodic joining $\lambda$ between $R_1\times R_2$ and $S_{1}\times
S_{2}\times S_{3}$.
In this joining we denote $\cd_i$ and $\cc_{j}$
the factor $\sigma$-algebra corresponding to each transformation $R_i$
and $S_{j}$.

By the JP property of $R_1$ we can assume, up to exchange of
coordinates,
that $\cd_1\vee \cc_1$ is independent from $\cc_2\vee\cc_3$, i.e.\
that $\cd_1\vee \cc_1$,
$\cc_2$ and $\cc_3$ are independent. Since $R_2$ also has the JP
property, $\cd_2$ is
independent from either $\cc_2$ or $\cc_3$ and we may assume
$\cd_2\perp\cc_3$.

Now, $\cd_1\vee\cc_1\vee\cc_2$ is independent from $\cc_3$ and, since
$R_2$ is distally simple,
by (\ref{jule}), either $\cd_1\vee\cd_2\vee\cc_1\vee\cc_2$ is
relatively distal
over $\cd_1\vee\cc_1\vee\cc_2$, or $\cd_2$ is independent from
$\cd_1\vee\cc_1\vee\cc_2$.
In the first case, since $S_3$ is weakly mixing, we get by
(\ref{furst}) that
$\cd_1\vee\cd_2\vee\cc_1\vee\cc_2$ remains independent from $\cc_3$.
In the latter case,
$\cd_1\vee\cc_1\vee\cc_2$, $\cd_2$ and $\cc_3$ are pairwise
independent, whence independent by the PID
property of $R_2$ and Lemma~\ref{lryz}, and we get again
$\cd_1\vee\cd_2\vee\cc_1\vee\cc_2\perp\cc_3$.

\end{proof}

\section{Flows with the convolution singularity property}\label{flows}
As noticed in the introduction,  all definitions and results of this
paper extend to the case of flows. 
In this section we will deal with flows since special flows over
irrational rotations can often be shown to satisfy the assumptions of
Proposition~\ref{francois} below.
Besides, it is not hard to see that a flow $(T_t)_{t\in\R}$ has the CS
property if and only if any given non-zero time $t$ automorphism of it
has the CS property.

Assume that $(T_{t})_{t\in\R}$ is a
(measurable) measure-preserving flow of $\xbm$ and let $\sigma$ be
its reduced maximal spectral type. We also denote by $(T_{t})$ the
associated unitary one-parameter group on $L^{2}\xbm$. Then, given
$t\in \R$, $T_{t}$ on
$L^2_0$ corresponds spectrally to the multiplication by $s\mapsto
e^{2\pi its}$ on $L^2(\R,\sigma)$ and the operator $V$ on
$L^{2}\xbm$, which is a weak limit of a sequence
$\left(T_{t_{n}}\right)_{n\in\N}$ corresponds to the
multiplication by a function $\phi_V$ being  a weak-$*$ limit of
the sequence $(e^{2\pi it_{n}\cdot})$ in $L^{\infty}(\R,\sigma)$.

\begin{Prop}\label{francois}
Assume that there exists a sequence
$\left(T_{t_{n}}\right)_{n\in\N}$ converging weakly to $V$ which
is not of the form $cT_{s}$ and such that $\phi_{V}$ is analytic.
Then $\sigma$ is singular with respect to the convolution of any two
continuous
measures.
\end{Prop}
\begin{proof} Let us notice first that the result is purely
spectral: {\em If $\sigma\not\perp\la_1\ast\la_2$, where $\la_i$
are positive continuous Borel measures on $\R$, and if $e^{2\pi i
t_n\cdot}\to\phi$ weakly-$\ast$ in $L^\infty(\R,\sigma)$, where
$\phi$ is analytic, then $\phi$ must be a multiple of a
character.}

Since $e^{2\pi it_n\cdot}\to\phi_V$ in $L^\infty(\R,\sigma)$
implies the same weak-$\ast$ convergence in $L^\infty(\R,\eta)$
for each $\eta\ll\sigma$, w.l.o.g. we can assume that
$0<\sigma\ll\lambda_{1}\ast\lambda_{2}$ and moreover that
$\phi_{V}$ is not identically zero. By passing to a subsequence of
$(t_n)$ if necessary and thanks to the weak-$\ast$ compactness,
$e^{2\pi it_{n}\cdot}$ converges weakly-$\ast$ to $\psi_{1}$,
$\psi_{2}$ and $\psi$ in $L^{\infty}(\R,\lambda_{1})$,
$L^{\infty}(\R,\lambda_{2})$ and
$L^{\infty}(\R,\lambda_{1}\ast\lambda_{2})$ respectively, where
$\psi_i$ and $\psi$ are Borel functions which we assume to be
defined everywhere. We thus obtain the {}``generalized character''
property \beq\label{francois1}
\psi\left(s_{1}+s_{2}\right)=\psi_{1}\left(s_{1}\right)\psi_{2}\left(s_{2}\right)\eeq
for $\lambda_{1}\otimes\lambda_{2}$-almost all
$\left(s_{1},s_{2}\right)\in\R\times\R$. This can be found in
\cite{Ho-Me-Pa}, but we give a proof for completeness. For $f_i\in
L^\infty(\R,\la_i)$ we put $d\la_i'=f_i\,d\la_i$; we have $$
\int_{\R}e^{2\pi it_ns}\,d\la'_1\ast\la'_2(s)=
\int_{\R\times\R}e^{2\pi i
t_n(s_1+s_2)}f_1(s_1)f_2(s_2)\,d\la_1(s_1)d\la_2(s_2) $$ $$ =
\int_{\R}e^{2\pi it_ns_1}f_1(s_1)\,d\la_1(s_1)\int_{\R}e^{2\pi
it_ns_2}f_2(s_2)\,d\la_2(s_2).$$ Thus, by passing to the weak-$*$
limits
$$
\int_{\R}\psi(s)\,d\la_1'\ast\la_2'(s)=\int_{\R}\psi_1(s_1)\,d\la_1'(s_1)
\int_{\R}\psi_2(s_2)\,d\la_2'(s_2).$$  It follows that
$$\int_{\R\times\R}\psi(s_1+s_2)\,d\la_1'(s_1)d\la_2'(s_2)=
\int_{\R}\psi_1(s_1)\,d\la_1'(s_1)
\int_{\R}\psi_2(s_2)\,d\la_2'(s_2),$$ whence (\ref{francois1})
holds since functions of the form $f_1\otimes f_2$ generate a linearly
dense subset of $L^1(\R\times\R,\la_1\otimes\la_2)$.

Of course, we also have $\psi=\phi_{V}$ $\sigma$-almost
everywhere. Put $h=\frac{d\sigma}{d\lambda_{1}\ast\lambda_{2}}$
and define the measure $\tilde{\sigma}$ on $\R\times\R$ by
$d\tilde{\sigma}=\tilde{h}\,d\la_1\otimes\la_2$, where
$\tilde{h}\left(s_{1},s_{2}\right)=h\left(s_{1}+s_{2}\right)$.
Then by~(\ref{francois1}) and the definition of $\tilde{\sigma}$
\beq\label{francois2} \phi_{V}\left(s_{1}+s_{2}\right)=
\psi_{1}\left(s_{1}\right)\psi_{2}\left(s_{2}\right)\eeq for
$\tilde{\sigma}$-almost all $\left(s_{1},s_{2}\right)\in\R^2$.

Put
$C=\{(s_1,s_2)\in\R^2:\:\phi_V(s_1+s_2)=\psi_1(s_1)\psi_2(s_2)\}$.
Then $C$ is a Borel subset of $\R^2$. Since $\tilde{h}>0$ on a set
of positive $\la_1\ast\la_2$-measure, $\la_1\otimes\la_2(C)>0$.
There exist Borel subsets $A$ and $B$ of $\R$ such that
\begin{equation}\label{francois3}
\lambda_{1}\otimes\lambda_{2}\left(\left(A\times B\right)\cap
C\right)\geq\frac{3}{4}\lambda_{1}\left(A\right)\lambda_{2}\left(B\right)>0.
\end{equation}
Given $s\in B$ define
$$
A_{s}=\left\{ s_{1}\in A:\left(s_{1},s\right)\in C\right\}$$ and
let
$$
B'=\left\{ s\in
B:\lambda_{1}\left(A_{s}\right)>\frac{1}{2}\lambda_{1}\left(A\right)\right\}.
$$
In view of (\ref{francois3}), $\lambda_{2}\left(B'\right)>0$. If
we fix $s_{2}\in B^{\prime}$, the equality
$\phi_{V}\left(s_{1}+s_{2}\right)=\psi_{1}\left(s_{1}\right)\psi_{2}\left(s_{2}\right)$
is satisfied for all $s_{1}\in A_{s_{2}}$. But $\la_1(A_{s_2})>0$,
so $A_{s_2}$ is uncountable and since $\phi_{V}$ is analytic and
not identically zero,  we must have
$\psi_{2}\left(s_{2}\right)\neq0$. We have shown that
$\psi_2(s_2)\neq0$ for $s_2\in B'$.

Now fix $s_{2}=s^{\prime}-s^{\prime\prime}$ with $s^{\prime}$ and
$s^{\prime\prime}$ in $B^{\prime}$. For $s_{1}\in
A_{s^{\prime}}\cap A_{s^{\prime\prime}}$, we obtain
\beq\label{francois4} \frac{\phi_{V}\left(s_{1}+s^{\prime}\right)}
{\psi_{2}\left(s^{\prime}\right)}=\frac{\phi_{V}
\left(s_{1}+s^{\prime\prime}\right)}{\psi_{2}\left(s^{\prime\prime}\right)}.\eeq
But $\lambda_{1}\left(A_{s^{\prime}}\cap
A_{s^{\prime\prime}}\right)>0$ from the definition of
$B^{\prime}$, and in particular $A_{s'}\cap A_{s^{\prime\prime}}$
is uncountable. Both functions in~(\ref{francois4}) being analytic
functions of $s_{1}$, the equality holds for all $s_1\in\R$.
Replacing now $s_1$ by  $s_{1}-s^{\prime\prime}$, we obtain that
\begin{equation}\label{francois5}
\phi_{V}\left(s_{1}+s_{2}\right)=\phi_{V}\left(s_{1}\right)
\frac{\psi_{2}\left(s^{\prime}\right)}
{\psi_{2}\left(s^{\prime\prime}\right)}\end{equation} for every
$s_{1}\in\R$, and in particular by taking $s_1=0$,
$\phi_{V}\left(s_{2}\right)= \phi_{V}\left(0\right)
\frac{\psi_{2}\left(s^{\prime}\right)}{\psi_{2}\left(s^{\prime\prime}\right)}$.
But the set $B^{\prime}-B^{\prime}$ is still uncountable, so we
must have $\phi_{V}\left(0\right)\neq0$. It now follows
from~(\ref{francois5}) that
$$
\phi_{V}\left(s_{1}+s_{2}\right)= \frac{\phi_{V}\left(s_{1}\right)
\phi_{V}\left(s_{2}\right)}{\phi_{V}\left(0\right)}$$ which holds
for $s_2\in B'-B'$ and all $s_1\in\R$. Once $s_1$ is fixed the
number of $s_2$ being uncountable, we hence obtain that this
equality holds for all $(s_1,s_2)\in\R^2$.  We conclude that
$\frac{\phi_{V}}{\phi_{V}\left(0\right)}$ is a continuous
character of $\R$ and the proof is complete.
\end{proof}

In order to obtain natural examples of flows satisfying the
assumptions of Proposition~\ref{francois} let us first notice that
if $P$ is a probability Borel measure on $\R$ and we have the weak
convergence \beq\label{przyklad} T_{t_n}\to\int_{\R}T_t\,dP(t)\eeq
of Markov operators in $L^2\xbm$ then
$$
e^{2\pi it_n\cdot}\to \hat{P}(\cdot)\;\;\mbox{weakly-$\ast$ in
$L^\infty(\R,\sigma)$}$$ (see e.g.\ the proof of Cor. 5.2 in
\cite{Fr-Le2}). Now, for example if $P$ has a compact support and is
not a
Dirac measure, or if $P$ is a Gaussian distribution,
then $\hat{P}$ is an analytic function different
from a multiple of a character.

This is satisfied for many
examples of special flows over irrational rotations: see  e.g.\
\cite{Fa-Wi},
\cite{Fr-Le1}, \cite{Fr-Le2}, \cite{Ka}, \cite{Le-Pa}, \cite{Le-Wy},
\cite{Vo}.

\vspace{2ex}
Added in October 2009: It is proved in \cite{Ka-Le} that there are
CS systems which are not DS. In fact, in \cite{Ka-Le} it is proved
that typical automorphism $T$ has a non-trivial relatively weakly
mixing extension $\widehat{T}$ such that $\widehat{T}$ has CS
property.

\end{document}